\documentclass[11pt, a4paper]{article}
\usepackage{amsfonts}
\usepackage{mathrsfs}
\usepackage{latexsym}
\usepackage{xy}
\usepackage{amsmath,amssymb,amsthm}
\usepackage{color}
\xyoption{all}

\newcommand{\bcen}{\begin{center}}     \newcommand{\ecen}{\end{center}}
\newcommand{\bay}{\begin{array}}      \newcommand{\eay}{\end{array}}
\newcommand{\beq}{\begin{eqnarray*}}      \newcommand{\eeq}{\end{eqnarray*}}

\def\ga{\alpha}
\def\gb{\beta}
\def\gs{\sigma}

\def\HH{\mathrm{HH}}
\def\HC{\mathrm{HC}}
\def\Hom{\mathrm{Hom}}

\def\Gr{\mathrm{Gr}}
\def\Grb{\mathrm{Gr}_{\bullet}}
\def\gr{\mathrm{gr}}
\def\Ext{\mathrm{Ext}}

\def\Mod{\mathrm{Mod}}

\def\csh{\mathrm{csh}}
\def\sh{\mathrm{sh}}
\def\Tor{\mathrm{Tor}}

\def\bz{\mathbb{Z}}

\def\cal{\mathcal}
\def\cf{\mathcal{F}}
\def\cm{\mathcal{M}}
\def\cn{\mathcal{N}}
\def\cg{\mathcal{G}}

\def\cs{\mathcal{S}}

\def\cz{\mathcal{Z}}
\def\fs{\mathfrak{S}}

\def\sa{\mathscr{A}}
\def\sb{\mathscr{B}}
\def\sc{\mathscr{C}}
\def\sn{\mathscr{N}}

\def\rr{\rightarrow}
\def\lrr{\longrightarrow}

\numberwithin{equation}{section}
\newtheorem{theorem}[equation]{Theorem}
\newtheorem{lemma}[equation]{Lemma}
\newtheorem{proposition}[equation]{Proposition}
\newtheorem{corollary}[equation]{Corollary}

\theoremstyle{definition}
\newtheorem{definition}[equation]{Definition}
\newtheorem{example}[equation]{Example}
\newtheorem{remark}[equation]{Remark}

\def\Number#1{\refstepcounter{equation}
              \leqno(\theequation)\if*#1%
              \else\def\@currentlabel{{\rm\theequation}}\label{#1}%
              \fi}
\newenvironment{point}[2]%
  {\vspace{0.5\jot}\ifx*#2\let\pointlabel\relax\else\def\pointlabel{#2}\fi
   \refstepcounter{equation}\trivlist
   \item[\bf\hskip\labelsep\theequation.
         \ifx\pointlabel\relax\else\space\pointlabel\space\fi]
   \ignorespaces #1
  }{\relax}
\fontsize{12}{\baselineskip}\selectfont
\begin{document}\vspace{-3truecm}
\title{Hochschild and cyclic (co)homology of superadditive categories}

\author{Deke Zhao}

\date{\footnotesize School of Applied Mathematics, \\ Beijing Normal University at Zhuhai,\\ No.\,18 Jinfeng Road,\\ Zhuhai City 519087, China\\e-mail:deke@amss.ac.cn}

\maketitle

\begin{abstract} We define the Hochschild
and cyclic (co)homology groups for superadditive categories and show that these (co)homology groups are graded Morita invariants.  We also show that the Hochschild and cyclic homology are compatible with the tensor product of superadditive categories.
\end{abstract}

{\footnotesize \noindent \textbf{Keywords:} $K$-categories; Superadditive categories; Hochschild (co)homology; Cyclic (co)homology; Graded Morita equivalence; K\"{u}nneth formula

\vspace{2\jot}
}

\section{Introduction}
Let $K$ be a field. Recall that a $K$-linear additive category (or simply $K$-category) is a category together with a $K$-vector space structure on each of its homomorphism sets such that composition is bilinear, which is defined and investigated by Mitchell in \cite{Mitchell}. Since their appearance, they have been used as a very important tool not only in algebra, but in many other fields, including algebraic topology, logic, computer science, etc. The Hochschild homology and cohomology theories $\HH_*(\sc,\cm)$ and $\HH^*(\sc,\cm)$ of a $K$-category $\sc$ with coefficients in a bimodule $\cm$ over $\sc$ were introduced by Mitchell in \cite{Mitchell}, and are closely related with theories studied by Keller \cite{Keller1999} and McCarthy \cite{McMarthy}. It is well-known that in case of a $K$-category with finite number of objects, the Hochschild (co)homology coincides with the usual Hochschild (co)homology of the $K$-algebra associated to the $K$-category (ref.~\cite[Proposition~2.7]{cibils-Redondo}).

A ($K$-linear) superadditive category is a $K$-category $\sa$ with each morphism set is a $\bz_2$-graded $K$-vector space (or simply superspace) such that the composition of morphisms is compatible with this $\bz_2$-grading.  The purpose of this paper is to define the Hochschild and cyclic (co)homology of superadditive categories and investigate their
behaviors with respect to the graded Morita equivalence and tensor product of superadditive categories.

 The Hochschild and cyclic (co)homology  of superadditive categories enjoy a number of desirable properties, the most basic being the agreement properties, i.e. the facts that if the superadditive category is trivially grading they coincides with the Hochschild and cyclic (co)homology theories of $K$-categories and that when applied to the superadditive category with a finite number of objects they specialize to the Hochschild and cyclic (co)homology of the corresponding superalgebra of the superadditive category (Proposition~\ref{Prop:Category-Algebras-HH-HC}). In particular, when the superadditive category is the superadditive category of finitely generated projective modules over a finite dimensional superalgebra they are exactly the  Hochschild and cyclic (co)homology of this superalgebra. Furthermore, we show that the Hochschild and cyclic (co)homology of superadditive categories are graded Morita equivalent invariants (Theorem~\ref{Them:HH-HC-Morita}) and prove the Eilenberg-Zilber Theorem for Hochschild homology and the K\"{u}nneth exact sequence for cyclic homology of superadditive categories (Theorems~\ref{Them:HH-shuffle} and \ref{Them:HC-Cyclic-shuffle}).

The paper is organized as follows. We begin in Section~\ref{Sec:Superadditive-cat} with preliminaries on superadditive categories and fix our notations. In Section~\ref{Sec:Hochschild-Cyclic} we define the Hochschild and cyclic (co)homology of superadditive categories and show that they are coincide with the usual Hochschild and cyclic (co)homology of superalgebras. In Section~\ref{Sec:Morita} we give a brief description of graded Morita equivalence of superadditive categories and prove the graded Morita equivalent invariance of the Hochschild and cyclic (co)homology. In Section~\ref{Sec:Shuflle-Cyclic-Shuffle}, we define the shuffle and cyclic shuffle product for superadditive categories and prove the Eilenberg-Zilber Theorem for Hochschild homology and the K\"{u}nneth exact sequence for cyclic homology.

\subsection*{Acknowledgements.}The author would like to thank the Institute of  Mathematics of the Mathematics Institute of the Chinese Academy of Sciences in Beijing and the Chern Institute of Mathematics in Nankai University for their hospitality and support while part of this work was carried out. The author was supported by the National Natural Science Foundation of China (Grant No. 11101037).

\section{Preliminaries on Superadditive categories}\label{Sec:Superadditive-cat}
In this section, we recall some facts on  superadditive category and  fix the notations.

\begin{point}{}* Throughout this paper we denote by $K$ a filed and by $\bz_2=\{0, 1\}$. All unadorned tensor products will be over the field $K$, i.e., $\otimes=\otimes_K$.  The composition $\ga\gb$ of two morphisms is to be read as first $\ga$ and then $\gb$. However, this will not present us from sometimes writing $\ga(x)$ in place of $x\ga$ when $\ga$ happens to be a function. When this is done, one must take into the account the switch $(\ga\gb)(x)=\gb(\ga(x))$.

In certain case when there can be no confusion, we shall write $X\in \sa$ to denote that $X$ is an object of $\sa$. Occasionally, we shall identify an object $X$ with its identity morphism $1_X$. For two objects $X, Y$ of a category $\sa$, we shall write $\sa(X,Y)$ for the set of morphisms form $X$ to $Y$ and write $_X\!\ga_Y$ for $\ga\in\sa(X,Y)$ to indicate the subscript.
 By a category we always mean a $K$-category, that is every Hom-set in it is a
$K$-linear space and the composition maps are bilinear. \end{point}

\begin{point}{}*\label{Point:Superspace}
Recall that a \textit{superspace} is a $\mathbb{Z}_2$-graded
$K$-vector space, namely a $K$-vector space $V$ with a decomposition
into two subspaces $V = V_{0}\oplus V_{1}$. A nonzero element $v$ of
$V_i$ will be called \textit{homogeneous} and we denote its degree by
$|v|=i\in \mathbb{Z}_2$. We will view  $K$ as a superspace concentrated in degree 0.
Given   superspaces $V$ and $W$, we view the direct sum $V\oplus W$
and the tensor product $V\otimes W$  as superspaces with $(V\oplus
W)_i = V_i\oplus W_i$, and $(V\otimes W)_i= V_0\otimes W_i\oplus
V_1\otimes W_{1-i}$  for $i\in\mathbb{Z}_2$. With this grading,
$V\otimes_K W$ is called the \textit{graded tensor product} of $V$ and
$W$.  Also, we make the
vector space $\Hom_K(V, W)$ of all $K$-linear maps from $V$ to $W$
into a superspace by setting that $\Hom_K(V, W)_i$ consists of all
the $K$-linear maps $f: V \rightarrow W$ with $f(V_j)\subseteq
W_{i+j}$ for $i, j\in \mathbb{Z}_2$. Elements of
$\Hom_K(V, W)_0$ (resp.~$\Hom_K(V, W)_1$) will be referred to as
\textit{even (resp.~odd) linear maps}. We denote by $\Gr K$ the category  whose objects are superspaces over $K$ and morphisms are even linear maps. Now the Koszul sign rule,
$v\otimes w\mapsto (-1)^{|v||w|}w\otimes v$ (Note this and other such expressions only make sense for homogeneous elements.), implies that $V\otimes W\cong W\otimes V$ in $\Gr K$.
The \textit{parity charge functor} $\Pi$ from $\Gr K$ to itself is an auto-equivalence defined as follows:
\begin{align*}
  (\Pi V)_{i} = V_{i+1} \text{ for }V=V_0\oplus V_1 \in\Gr K \text{ and }i\in\bz_2;\qquad
  \Pi(\ga)=\ga \text{ for }\ga\in \Gr K.
\end{align*}
 \end{point}
\vspace{-0.5truecm}
\begin{definition}[\protect{\cite[Chapter~3.2.7]{Manin}}]
A \textit{superadditive category} $\sa$ is a $K$-category with each morphism set $\sa(X,Y)$ is provided a $\bz_2$-grading compatible with composition of morphisms, that is, for each pair $X,Y\in \sa$ we have \begin{equation*}
 \sa(X,Y)=\sa(X,Y)_0\oplus\sa(X,Y)_1
\end{equation*}such that if $\ga\in\sa(X,Y)_i$ and $\gb\in\sa(Y,Z)_j$, then $\ga\gb\in\sa(X,Z)_{i+j}$ for $i,j\in\bz_2$ and $Z\in \sa$.
 As usual,  $\sa^{\mathrm{op}}$ denotes  the \textit{opposite category} of $\sa$, which is the superadditive category has the same objects and morphisms as $\sa$ but the composition of morphisms is given by $$\phi^{\mathrm{o}}\psi^\mathrm{o}:=(-1)^{|\phi||\psi|}(\psi\phi)^{\mathrm{o}}.$$
 \end{definition}

\begin{example}\label{Exam:shifts}Let $A$ be a finite dimensional superalgebra over
$K$ and denote by $\mathrm{Gr\,}A$
 the category of all (left) $\bz_2$-graded $A$-modules with even
 homomorphisms $$\Hom_{\Gr A}(X, Y)=\{f\in\Hom_{A}(X,Y)|f(X_i)\subseteq Y_i \text{ for all }i\in
 \bz_2\}.$$ Denote by $\gr A$ the full subcategory of $\Gr A$ consisting of all finitely generated $G$-graded $A$-modules. Then $\Gr A$  and $\gr A$  are abelian categories.
   Let $\Grb A$ be the category with the same objects as $\Gr A$ but with
 extended $\bz_2$-graded Hom-sets
 $$\Hom_{\Grb A}(X,Y):=  \Hom_{\Grb A}(X,Y)_0\oplus \Hom_{\Grb A}(X,Y)_1
 =\Hom_{\Gr A}(X,Y)\oplus\Hom_{\Gr A}(X,\Pi Y),$$ where $\Pi$ is the \textit{parity change functor}  from $\Gr A$ to itself (ref.\,\cite[Chapter~3]{Manin}).
Then $\Grb A$ is a superadditive category. Similarly, we can obtain the superadditive category $\gr_{\bullet} A$.
\end{example}

\begin{point}{}* Let $\sa$ and $\sb$ be superadditive categories. A
(covariant $\bz_2$-)\textit{graded functor} $\cf: \sa\rr\sb$
is an additive functor between the (ungraded) categories
$\sa$ and $\sb$ such that $\cf$ induces an even linear map
between $\sa(-,-)$ and
$\sb(\cf(-),\cf(-))$; that is, for all $X, Y\in\sa$,
\begin{equation*}\cf_{XY}:\, \sa(X,Y)=\bigoplus_{i\in \bz_2} \sa(X,Y)_i\longrightarrow \sb(\cf(X),\cf(Y))=\displaystyle{\bigoplus_{i\in
\bz_2}}\sb(\cf(X),\cf(Y))_i\end{equation*}
  is an even linear map,  or equivalently, there is a family $\{M(X)\}_{X\in\sa}$
  of  superspaces together with even linear map $$\sa(X,Y)\otimes M(Y)\rightarrow M(X)$$
  such that $1_{Z}\cdot m=m$,  $\phi\cdot (\psi\cdot m)=(\phi\psi)m$ for all $m\in M(Z)$,
  $\phi\in \sa(X,Y)$, $\psi\in \sa(Y,Z)$. $\sa(X, -): \sa\rr\Gr K$ is a covariant $\bz_2$-graded functor for each $X\in \sa$.

 \begin{remark}
   Let $\cs$ be the suspension form $\sa$ to itself, which is defined by $\cs(X)=X$ and $\cs(\ga_i)=\ga_{i+1}$ for $X\in \sa$, $\ga\in\sa(-,-)$, and $i\in\bz_2$. Then a functor $\cf: \sa\rr\sb$ is a graded functor provided $\cf$ commutes with the suspension.
\end{remark}

 Let $\cf,\cg:\sa\rr \sb$ be graded functors. Then a natural transformation $\eta: \cf\rr\cg$ is said to \textit{even} if it commutes with the suspension $\cs$. We denote by $\Gr(\sa, \sb)$ the category of all graded functors from $\sa$ to $\sb$ and morphisms being the graded natural transformations.

 By a \textit{graded (left) $\sa$-module} we means a graded functor $\cm:\sa\rr\Gr K$. We denote by $\Gr \sa$ the category of all graded $\sa$-modules and
morphisms being the graded natural transformations.
By definition, a ($\bz_2$)-{\it graded $\sa$-bimodule} is a ($\bz_2$-)graded bifunctor $\cm: {\sa} \times {\sa}^{\mathrm{op}}\rr\Gr K$. In other words $\cm$ is given by a set of
superspaces $\{ \cm(X,Y)\}_{X,Y \in {\sa}}$
and left and right actions
\begin{align*}
\cm(X,Y)\otimes\sa(Y,Z)\longrightarrow\cm(X,Z)\quad
\text{ and }\quad
\sa(Z,X)\otimes\cm(X,Y)\longrightarrow\cm(Z,Y)
\end{align*}
satisfying the usual associativity conditions for all $X,Y,Z\in\sa$. For instance the
standard graded bimodule over $\sa$ is itself.

\begin{point}{}* The \textit{(naive) tensor product category} $\sa\otimes\sb$ of superadditive categories $\sa$ and $\sb$ is the category whose objects are $X\otimes Y$ for $X\in\sa$ and $Y\in\sb$ and whose morphisms from $X\otimes Y$ to $X'\otimes Y'$ is the graded tensor product of superspaces $\sa(X,X')\otimes \sb(Y,Y')$. The composition is given by the rule \begin{align*}(\phi_1\otimes\phi_2)(\psi_1\otimes\psi_2)=
(-1)^{|\phi_2|\psi_1|}\phi_1\psi_1\otimes\phi_2\psi_2.\end{align*} Clearly $\sa\otimes\sb$ is again a superadditive category and we have the following canonical isomorphisms of superadditive categories $$(\sa\otimes\sb)\otimes\sc\simeq\sa\otimes(\sb\otimes\sc), \qquad\quad \sa\otimes\sb\simeq\sb\otimes\sa.$$

The category $\sa^e:=\sa\otimes\sa^{\mathrm{op}}$ will be called the \textit{enveloping category} of $\sa$. Then graded $\sa$-bimodules are equivalent to graded $\sa^e$-modules, which are given by $$\phi\mathfrak{m}\psi=(-1)^{|\mathfrak{m}|\psi}(\phi\otimes\psi)\mathfrak{m}=
(-1)^{|\phi|(|\mathfrak{m}|+|\psi|)}\mathfrak{m}(\psi\otimes\phi).$$
\end{point}

Recall that the parity charge $\Pi$ on $\Gr K$ defined in Section~\ref{Point:Superspace}. Given a graded $\sa$-module $\cf$ we define a shift operation $\pi$ on
 $\cf$ by letting $$\pi(\cf)(X):=\Pi\cf(X) \text{ for }X\in\sa\text{ and } \pi(\cf)(\ga):=\Pi\cf(\ga)\text{ for morphism }\ga\in\sa.$$
Now we define the category $\Gr_{\bullet} \sa$
to be the category with the same objects as $\Gr \sa$ but with the extended ($\bz_2$-graded) Hom-sets $$\Hom_{\Gr_{\bullet} \sa}(\cf, \cg):=\Hom_{\Gr\sa}(\cf, \cg)\bigoplus\Hom_{\Gr\sa}(\cf, \pi(\cg))\text{ for }\cf, \cg\in\Gr\sa,$$  alternatively, $\Grb\sa$ is the category
of all graded functors $\cf:\sa\rr\Gr K$ and morphism being the
natural transformation $\eta: \cf\rr\cg$ with
$\eta_X:\cf(X)\rr\cg(X)$ is a linear map for each $X$ in $\sa$. In
this way $\Grb \sa$ becomes a superadditive category.
\end{point}

\begin{remark}\label{Remark:superadditive-algebra}(i) If $\sa$ is a    superadditive category with finite number of objects,
 then \textit{the corresponding superalgebra} of $\sa$ is the unital superalgebra
$\Lambda_{\sa}:= \bigoplus_{X,Y \in {\sa}}
\sa(X,Y)$ with
a well-defined matrix product given by the composition of $\sa$ and  the
identity is the diagonal matrix with $1_{X}$ in
the diagonal. Furthermore, the usual $\bz_2$-graded
$\Lambda_{\sa}$-(bi)modules and $\bz_2$-graded $\sa$-(bi)modules coincide.

(ii) Let $A$ be a superalgebra equipped with a finite
complete set $E$ of orthogonal idempotents (non necessarily
primitive). Then there is a  finite superadditive category $\sc(A, E)$ associated to
$(A,E)$ with objects $E$ and morphisms form $x$ to $y$ is $xAy$.
Composition is given by the product in $A$. Furthermore, the
$\bz_2$-graded $\sc(A,E)$-(bi)modules and $\bz_2$-graded $A$-(bi)modules
coincide.
\end{remark}

\vspace{-0.5truecm}
\section{Hochschild and cyclic (co)homology}\label{Sec:Hochschild-Cyclic}
In this section we define Hochschild and cyclic (co)homology groups for superadditive categories by applying the ($\bz_2$-graded) Hochschild-Mitchell complex. Throughout this section, $\sa$ is a superadditive category unless otherwise stated.

\begin{definition}Let $X(n):=(X_{n}, \dots , X_1, X_0)$ be an $n+1$-tuple of objects of $\sa$. The {\it $K$-nerve} associated to $X(n)$ is the superspace \begin{align*}\sa(X_{n},X_{n-1}) \otimes \cdots \otimes\sa(X_{1},X_{0}).\end{align*}  The {\it $K$-nerve $\mathscr{N}_n$ of
degree $n$} is the direct sum of all the $K$-nerves associated to
$n+1$-tuples of objects
\begin{eqnarray*}
\sn_n(\sa) & = & \bigoplus_{n+1\text{-tuples}} \sa(X_{n},X_{n-1})
\otimes \dots \otimes \sa(X_{1},X_{0}).
\end{eqnarray*}
An element $a\in\sn_n(\sa)$ is a chain of degree $n$ and denote it by $a=(a_0,a_1,\cdots,a_n)$.
Then $\sn_n(\sa)$ is a graded $\sa$-bimodule defined by
\begin{align*}\sn_n(\sa)(X,Y):=\bigoplus_{n+1\text{-tuples}}\sa(X,X_n)\otimes\sa(X_{n},X_{n-1})
\otimes \dots \otimes \sa(X_{1},X_{0})\otimes\sa(X_0,Y).\end{align*}
The ($\bz_2$-graded) \textit{Hochschild-Mitchell complex} or \textit{standard complex} $\sn(\sa)$ of $\sa$ is
\begin{align*}
   \cdots \stackrel{\partial_{n+1}}\longrightarrow \sn_n(\sa) \stackrel{\partial_n}{\longrightarrow}
\cdots \stackrel{\partial_2}{\longrightarrow}\sn_1(\sa)
\stackrel{\partial_1}{\longrightarrow}\sn_0(\sa)
\stackrel{\partial_{0}}\longrightarrow \sa\longrightarrow 0.
\end{align*}
where $\partial_n=\sum_{i=0}^n(\!-1\!)^id_n^i$ with $d_n^i: \sn_n(\sa)\rr\sn_{n-1}(\sa)$ defined as follows:
\begin{equation}\label{Equ:Standard-complex}d_n^i(\ga_0\otimes\ga_1\otimes\cdots\otimes\ga_{n})=\left\{\aligned&
a_0a_1\otimes\cdots\otimes a_{n}&& \text{ if }
i=0\\&a_0\otimes\cdots\otimes a_i a_{i+1}\otimes\cdots\otimes a_{n}&&\text{ if }1\le i<n\\
&(-1)^{| a_n|| a_0a_1\cdots  a_{n-1}|} a_n a_0\otimes a_1\otimes\cdots\otimes a_{n-1}
&&\text{ if }i=n.\endaligned\right.
\end{equation}
Note that each $d_n^i$ is well-defined and $\partial_{n+1}\partial_{n}=0$. This complex is a projective resolution of the standard graded $\sa$-bimodule $\sa$.
\end{definition}

\begin{definition} Let $\cm$ be a graded $\sa$-bimodule. The \textit{Hochschild
homology groups} $\HH_*({\sa},{\cal M})$ of $\sa$ with
coefficients in $\cm$ are defined to the homology groups of the chain complex
\begin{align*} \cdots \stackrel{\partial_{n+1}}\longrightarrow C_n(\sa,\cm) \stackrel{\partial_n}{\longrightarrow}
\cdots \stackrel{\partial_2}{\longrightarrow}C_1(\sa,\cm)
\stackrel{\partial_1}{\longrightarrow}C_0(\sa, \cm)
\longrightarrow 0,\end{align*} where $\partial$ is defined as (\ref{Equ:Standard-complex}) and
\begin{align*}
  C_n(\sa,\cm)=\cm\otimes_{\sa^e}\sn_n(\sa)=\bigoplus_{n+1\text{-tuples}}\cm(X_0,X_n)\otimes\sa(X_{n},X_{n-1})
\otimes \dots \otimes \sa(X_{1},X_{0}).
\end{align*}
\end{definition}
\vspace{-0.4truecm}
\begin{remark}\label{Rem:HH_0} If $\sa$ is trivially graded, then the Hochschild homology groups $\HH_*(\sa,\sa)$ coincide with the ones defined in the ungraded situation (ref.~\cite[\S~17]{Mitchell}). One always has \begin{equation*}
  \HH_0(\sa,\sa)=\sa/[\sa, \sa]_{\gr},
\end{equation*} where $[\sa,\sa]_{\gr}$ is the subsuperspace spanned by all graded commutators \begin{align*}[\ga, \gb]_{\gr}=\ga\gb-(-1)^{|\ga||\gb|}\gb\ga\qquad  \text{ for morphisms }\ga, \gb \text{ of }\sa.\end{align*} As in the ungraded case, the Hochschild homology of a superadditive category can be defined as a derived functor. Indeed, $\HH_*(\sa, \sa)=\Tor_*^{\sa^e}(\sa, \sa)$.\end{remark}

 \begin{definition} Let $\cm$ be a graded $\sa$-bimodule. The \textit{Hochschild
cohomology groups} $\HH^*({\sa},{\cal M})$ of $\sa$ with
coefficients in $\cm$ are defined to the cohomology groups of the cochain complex
\begin{align*} 0 \longrightarrow C^0(\sa,\cm)\stackrel{\partial^0}{\longrightarrow}
C^1(\sa,\cm)\stackrel{\partial^1}{\longrightarrow}
\dots \stackrel{\partial^{n\!-\!1}}{\longrightarrow} C^n(\sa,\cm) \stackrel{\partial^n}{\longrightarrow} \dots, \end{align*}
where \begin{align*}C^n(\sa,\cm)=\Hom(\sn_n,\cm)= \prod_{n\!+\!1\text{-tuples}} \Hom_K \left(\sa(X_{n},X_{n-1})\otimes \dots \otimes \sa(X_{1},X_{0}),{\cal M}(X_{n},
X_0)\right)\end{align*} and $\partial^n=\sum_{i=0}^n(-1)^id_i^n$ with $d_n^i: C^{n}(\sa, \cm)\rr C^{n\!+\!1}(\sa,\cm)$ defined as following:
$$(d_n^i\phi)(a_1, \cdots, a_{n+1})=\left\{\aligned&
(-1)^{|a_1||\phi|}a_1\phi(a_2, \cdots, a_{n+1})&& \text{ if }
i=0\\&\phi(a_1,\cdots, a_ia_{i+1},\cdots, a_{n+1})&&\text{ if }1\le i<n\\
&\phi(a_1, \cdots, a_{n})a_{n+1} &&\text{ if }i=n.\endaligned\right. $$
Note that each linear map of the family $d_n^i(\phi)$ is well-defined for $\phi\in C^n(\sa,\cm)$ and $\partial^n\partial^{n+1}=0$.
\end{definition}
Recall that the graded center $\cz_{\gr}(\sa)$ of $\sa$ is the ring of
endomorphisms of the identity functor to itself. More
concretely $\cz_{\gr}(\sa)$ is a graded ring whose homogeneous
elements consist of tuples of homogeneous elements $(\phi_X)_X$ with $X\in \sa$ and
$\phi_X\in\sa(X,X)$ such that for any homogeneous
$\psi\in\sa(X, Y )$ one has $\phi_X\psi = (-1)^{|\psi
||\phi_X|}\psi\phi_Y$. For example, if $A$ is a finite dimensional
superalgebra then $\cz_{\gr}(\gr_{\bullet} A)$ is canonically
isomorphic to the graded center  of the superalgebra $A$ itself, that is,
$\cz_{\gr}(\gr_{\bullet} A)\negmedspace:=\mathrm{Span}_K\{a\in
A|ab=(-)^{|a||b|}ba\, \text{ for all }b\in A\}$.
\begin{point}{}* \label{Point:HH^0}
It is clear that
$$\HH^0({\sa}, {\cal M})\!=\! \{ (\mathfrak{m}_X)_X\!\mid\!
\mathfrak{m}_X \!\in\! \cm(X,X) \text{ and }\phi\mathfrak{m}_X\!=\!(\!-1\!)^{|\mathfrak{m}_X||\phi|}
\mathfrak{m}_Y \phi \text{ for all }\phi\!\in\!\sa(Y,X), Y\!\!\in\!\sa\}.$$
Hence $\HH^0(\sa, \sa)=\cz_{\gr}(\sa)$.
Note that the standard bimodule $\sa$ has a
resolution by tensor powers of $\sa$ which are projective graded
bimodules. Applying the functor $\Hom (-, {\cal M})$ provides
precisely the cochain complex above. Consequently
Hochschild cohomology is an instance of an $\Ext$
functor, namely $\HH^{*}(\sa,\sa)=\Ext^*_{\sa^e}(\sa,\sa)$.
\end{point}

 \begin{point}{}*Now let $t_n$ be the generator of $\bz/(n+1)\bz$  and define
its action on $\sn_{n+1}(\sa)$ by \begin{align*}t_n(a_0, a_1,
\cdots, a_n)=(-1)^{n+|a_n|(|a_0|+\cdots+|a_{n-1}|)}(a_n, a_0,
\cdots, a_{n-1}).\end{align*}Then  for all $n\ge 1$, the operators
$t_n$, $N_n:=1+t_n+\cdots+t_n^n$, $\partial_n$, and $\bar{\partial}_n:=\sum_{i=0}^{n-1}d_n^i$
satisfy identities $t_n^{n+1}=\mathrm{id}$,
$(1-t_{n-1})\bar{\partial}_n=\partial_n(1-t_n)$, and
$\bar{\partial}_nN_n=N_{n-1}\partial_n$.\end{point}

\begin{point}{}*\label{Point:Cyclic-bicomplex} Now we get the
 following \textit{cyclic bicomplex} $CC_{**}(\sa)$
 \begin{align*}&\xymatrix@C=10mm{
  \vdots\ar[d]_{\partial_2}&\vdots\ar[d]_{-\bar{\partial}_2}&\vdots\ar[d]_{\partial_2}&\vdots\ar[d]_{-\bar{\partial}_2}\\
  \sn_1(\sa)\ar[d]_{\partial_1}& \sn_1(\sa)\ar[d]_{-\bar{\partial}_1}\ar[l]_{1-t_2}
  & \sn_1(\sa)\ar[d]_{\partial_1} \ar[l]_{N_2} &\sn_1(\sa) \ar[d]_{-\bar{\partial}_1}\ar[l]_{1-t_2}&\cdots\ar[l]_{N_2}\\
 \sn_0(\sa)\ar[d]_{\partial_0}& \sn_0(\sa)\ar[d]_{-\bar{\partial}_0}\ar[l]_{1-t_1}
  & \sn_0(\sa)\ar[d]_{\partial_0} \ar[l]_{N_1} &\sn_0(\sa) \ar[d]_{-\bar{\partial}_0}\ar[l]_{1-t_1}&\cdots\ar[l]_{N_1}\\
  \sa& \sa\ar[l]_{1-t_0}
  &  \sa\ar[l]_{N_0} &\sa\ar[l]_{1-t_0}&\cdots\ar[l]_{N_0}\\}\end{align*}
   By convention, the standard graded bimodule $\sa$ being in the
left-hand corner is of bidegree $(0, 0)$, so
$CC_{pq}(\sa)=\sn_{q}(\sa)$. Note that columns of $CC_{**}(\sa)$ are copies of the
Hochschild-Mithcell complex $\sn(\sa)$ of $\sa$. The cyclic complex $C(\sa)$ is by
definition the total complex of this double complex, that is, the space
of cyclic $n$-chains $ C(\sa)_n =\bigoplus_{k\geq 0}
\sn_{n-2k}(\sa)$ with \textit{cyclic boundary} $\partial+ B$  where
$B=(1-t)s\sum_{i=0}^nt^i$ and $s$ is the \textit{extra degeneracy} defined by \begin{align*}s(\ga_0\otimes\ga_1\otimes\cdots\otimes\ga_n)=1\otimes\ga_0\otimes\ga_1\otimes\cdots\otimes\ga_n.\end{align*}
\end{point}
\begin{definition} The  \textit{cyclic homology groups}
$\HC_*(\sa)$ of the superadditive category $\sa$ is defined to be the
homology groups of the cyclic complex $C(\sa)$, i.e.,
$\HC_n(\sa):=\mathrm{H}_n(\mathrm{Tot}\,CC(\sa))$.\end{definition}

Now the arguments of \cite[\S\,1]{Loday-Quillen} can be adopted to the case of superadditive categories.  More precisely  there is a so-called \textit{Connes periodicity operator}
\begin{align*}S : C_n(\sa) \rr C_{n-2}(\sa), \,(\ga_n, \ga_{n-2},
\ga_{n-4}, \cdots)\mapsto(\ga_{n-2}, \ga_{n-4}, \cdots) \text{ where }
\ga_i\in \sn_i(\sa).\end{align*} Let $I$ be the inclusion of the Hochschild-Mitchell complex as
the first column of the cyclic complex. Then the Connes periodicity operator induces the following short exact sequence of complexes \begin{equation}\xymatrix@C=0.5cm{
  0 \ar[r] & \sn(\sa) \ar[r]^{I} &C(\sa) \ar[r]^{S} & C(\sa)[2]\ar[r]&0 }\end{equation}

Applying the snake Lemma, we obtain the following \textit{Gysin-Connes exact sequence} \begin{equation}\label{Equ:Gysin-Connes-Exact-Sequ} \xymatrix@C=0.5cm{ \cdots \ar[r] &
\HH_n(\sa) \ar[rr]^{I} &&\HC_n(\sa)\ar[rr]^{S} && \HC_{n-2}(\sa) \ar[rr]^{B} &&
\HH_{n-1}(\sa) \ar[rr]^{I} && \cdots}.\end{equation}

We construct now the cyclic cohomology of a superadditive category.

\begin{point}{}* Dualizing the cyclic bicomplex $CC_{**}(\sa)$ introduced in Section~\ref{Point:Cyclic-bicomplex}, we obtain a bicomplex of cochains $CC^{**}(\sa)$ such that $CC^{pq}(\sa)=C^{q}(\sa)$. It has vertical differential maps $\partial^*$ or $\bar{\partial}*: CC^{pq}(\sa)\rightarrow CC_{pq+1}(\sa)$ and horizontal differential maps $(1-t)^*$ or $N^*: CC^{pq}(\sa)\rightarrow CC_{p+1q}(\sa)$. By definition the cyclic cohomology groups of $\sa$ is the cohomology groups of the cochain complex $\mathrm{Tot\,}CC^{**}(\sa)$, that is, $\HC^{n}(\sa):=H^{n}(\mathrm{Tot\,}CC^{**}(\sa))$.
\end{point}

Note that if $\sa$ has just one object, then the Hochschild and cyclic (co)homology theories of $\sa$ are the usual Hochschild and cyclic (co)homology theories of superalgebras (ref. \cite{kassel,Zhao}). More generally, we have the following facts.
\begin{proposition}\label{Prop:Category-Algebras-HH-HC}
Let $\sa$ be a superadditive category with finite number of objects and let $\Lambda_{\sa}$ be the
corresponding superalgebra. Then
\noindent\begin{enumerate}\item $\HH^*({\sa}, M) =\HH^*(\Lambda_{\sa}, M)$  and $\HH_*({\sa}, M) = \HH_*(\Lambda_{\sa}, M)$ where $\HH_*(\Lambda_{\sa}, M)$  and $\HH^*(\Lambda_{\sa}, M)$) are respectively  the Hochschild homology  and cohomology of  $\Lambda_{\sa}$ with coefficients in graded $\Lambda_{\sa}$-bimodule $M$.
\item $\HC^*({\sa}) =\HC^*(\Lambda_{\sa})$ and
$\HC_*({\sa}) = \HC_*(\Lambda_{\sa})$ where
$\HC_*(\Lambda_{\sa})$ and  $\HC^*(\Lambda_{\sa})$ are respectively the cyclic homology and cohomology of $\Lambda_{\sa}$.\end{enumerate}
\end{proposition}
\begin{proof}
Let $\Lambda := \Lambda_{\sa}$ and consider the semisimple
sub-superalgebra $E= \prod K1_X$ (concentrated in degree 0) of $\Lambda$. Note that
any $\bz_2$-graded $E$-bimodule is graded projective since the enveloping superalgebra of $E$
is still semisimple. Consequently there is a $\bz_2$-graded projective resolution of $\Lambda$ as
a graded $\Lambda$-bimodule given by
$$ \dots \longrightarrow \Lambda \otimes_E \Lambda \otimes_E \dots \otimes_E \Lambda
\longrightarrow \dots \longrightarrow \Lambda \otimes_E \Lambda
\longrightarrow \Lambda \longrightarrow 0.$$ Applying the functor
$M\otimes_{\Lambda^e}- $ to this resolution and
considering the canonical superspace isomorphism
$$M\otimes_{\Lambda^e}(\Lambda\otimes_E - \otimes_E\Lambda)=
M\otimes_E-,$$ we obtain a chain complex computing
$\HH_*(\Lambda, M)$ which coincides with the complex we have
defined for the Hochschild homology of superalgebras (see \cite[\S3.1]{Zhao}).

Applying the functor $\Hom_{\Lambda^e}(-,M)$ to this resolution and
considering the canonical superspace isomorphism
$$\Hom_{\Lambda^e}(\Lambda\otimes_E - \otimes_E\Lambda, M)=
\Hom_{E^e}(-,M),$$ we obtain a cochain complex computing
$\HH^*(\Lambda, M)$ which coincides with the complex we have
defined for the Hochschild cohomology of superalgebras (see  \cite[\S3.1]{Zhao}).
The assertion for cyclic (co)homology follows by applying the Gysin-Connes exact sequence (\ref{Equ:Gysin-Connes-Exact-Sequ}).
\end{proof}

\section{Graded Morita invariance}\label{Sec:Morita}

In this section  we give a brief description of the graded Morita theory for superadditive categories and prove the graded Morita invariance of the Hochschild and cyclic (co)homologies.

 Let $\sa$ be a superadditive category. Denote by $\Mod \sa$ the category  of all (ungraded) $\sa$-modules, that is, the category consists of all additive covariant functors from $\sa$ to $\Mod K$ and morphisms being the natural transformations. Denote by $\mathbb{F}$ the \textit{forgetful functor} from $\Gr\sa$ to $\Mod\sa$. Now let $\sb$ be another superadditive category. A linear additive functor
  $\cf\negmedspace:\Gr\sa \longrightarrow\Gr\sb$
 is an {\em even functor} provided it commutes with the charge parity $\pi$.
An even functor $\cf\negmedspace:\Gr\sa \longrightarrow\Gr\sb$ is said to be an \textit{even equivalence} provided there is an even functor $\cg\negmedspace:\Gr\sb \longrightarrow\Gr\sa$ such that $\cf\cg=\mathrm{Id}_{\Gr \sa}$ and $\cg\cf=\mathrm{Id}_{\Gr\sb}$. In this case we say that $\sa$ and $\sb$ are \textit{even equivalent}.

\begin{definition}Two superadditive categories $\sa$ and $\sb$ are said to be \textit{graded Morita equivalent} if  $\Gr\sa$ is even equivalent to $\Gr\sb$.
\end{definition}

Let $\sa$ be a superadditive category. \textit{An idempotent} in $\sa$ is an even morphism $\epsilon\!\in\!\sa(X,X)$ such that $\epsilon^2=\epsilon$. An idempotent $\epsilon\!\in\!\sa(X,X)$ is \textit{split} if there is an object $Y\in\sa$ and homogeneous morphisms  $\ga\in\sa(X,Y)$, $\gb\in\sa(Y,X)$ such that $\ga\gb=\epsilon$ and $\gb\ga=1_Y$.  We say that $\sa$ is \textit{idempotent complete} if every idempotent is split.

\begin{remark}Every superadditive category $\sa$ be fully faithfully embedded into an idempotent complete superadditive category $\widehat{\sa}$. \end{remark}

The following construction of $\widehat{\sa}$ is essential due to Freyd (cf.~\cite{Mitchell1965}):
  The objects of $\widehat{\sa}$ are the pairs $(X,\epsilon)$ consisting of an object $X$ of $\sa$ and an idempotent $\epsilon\in\sa(X,X)$ while the sets of morphisms are
  \begin{equation*}
    \widehat{\sa}\left((X,\epsilon), (Y,\varepsilon)\right)=\epsilon\sa(X,Y)\varepsilon
  \end{equation*}with composition induced by the composition in $\sa$. It is easy to see that $\widehat{\sa}$ is a superadditive category  with biproduct $(X, \epsilon)\oplus(Y,\varepsilon) = (X\oplus Y, \epsilon\oplus\varepsilon)$ and obviously the functor
$\mathrm{hat}_{\sa}: \sa\rr\widehat{\sa}$ given by $\mathrm{hat}_\sa(X)= (X, 1_X)$ and $\mathrm{hat}_{\sa}(\phi) =\phi$ is is a fully faithful graded additive functor. Furthermore  $\widehat{\sa}$ is idempotent complete, which is called the {\it idempotent complete} of $\sa$. Note that if $\sa$ is idempotent complete then $\mathrm{hat}_{\sa}: \sa\rr \widehat{\sa}$ is an equivalence of superadditive categories.

 The following facts can be proved by the similar arguments as those of \cite[\S6]{Buhler} for $K$-categories.

 \begin{proposition}\label{Prop:hat-equiv} Let $\sa$ and $\sb$ be superadditive categories. If $\sb$ is idempotent complete then $\mathrm{hat}_{\sa}$ induces an equivalence $\mathrm{hat}:\mathrm{Gr}(\widehat{\sa}, \sb)\stackrel{\sim}\longrightarrow\mathrm{Gr}(\sa, \sb)$. In particular, $\mathrm{hat}$ induces an equivalence of categories $\mathrm{hat}:\mathrm{Gr}\widehat{\sa}^e \stackrel{\sim}\longrightarrow\mathrm{Gr}\sa^e$.
 \end{proposition}

\begin{point}{}* \label{Point:Matrix-equ.}
Given any superadditive category $\sa$, we can form the category $\mathrm{Mat}\sa$ whose objects are finite sequence of objects in $\sa$ and where a morphism from $(X_1, \cdots, X_m)$ to $(Y_1, \cdots, Y_n)$ is an $m\times n$ matrix $[\ga_{ij}]$ with $\ga_{ij}\in\sa(X_i,Y_j)$, composition is defined by ordinary matrix multiplication. The category $\mathrm{Mat}(\sa)$ is called an \textit{additive completion} of $\sa$. Then $\mathrm{Mat}\sa$ is a superadditive category with finite product and contains $\sa$ as a full superadditive subcategory by identifying $X\in\sa$ with the matrix whose only nonzero entry is $1_X$. Moreover, the functor $\mathrm{mat}: \sa\rr\mathrm{Mat}\sa$ defined by $X\mapsto (X)$ is a graded functor, and if $\sb$ is small superadditive category then $\mathrm{mat}$ induces an equivalence $$\mathrm{mat}:\quad\mathrm{Gr}(\mathrm{Mat}\sa, \sb)\stackrel{\sim}\longrightarrow\mathrm{Gr}(\sa, \sb).$$
In particular, $\mathrm{mat}$ induces an equivalence of categories $\mathrm{mat}: \mathrm{Gr}(\mathrm{Mat}\sa)^e\stackrel{\sim}\longrightarrow\mathrm{Gr}\sa^e$.\end{point}

 By Proposition~\ref{Prop:hat-equiv} and Section~\ref{Point:Matrix-equ.},
two superadditive categories are graded Morita equivalent if and only if their completions (i.e., the additive completion and the idempotent completions) are graded Morita equivalent. Note that these completions are superadditive  categories with finite products, so they are graded Morita equivalent if and only if they are equivalent according to Section~\ref{Point:Matrix-equ.}.

Using the similar arguments as that of \cite{Cibils-Solotar}, we can prove a super-version of \cite[Theorem~4.7]{Cibils-Solotar}. The details will
appear elsewhere.

\begin{theorem}\label{Them:Morita-addi-idem}
  Any graded Morita equivalence between superadditive categories is a composition of equivalences, additivization and idempotent completions of superadditive categories.
\end{theorem}

Given two superadditive categories $\sa$ and $\sb$, a graded left
$\sb$-module $\cn$ and a graded functor $\cf\negmedspace: \sa\rr\sb$,
we define the graded $\sa$-module $\cf(\cn)$ given by
$\cf(\cn)(Z)=\cn(\cf(Z))$ for all $Z\in \sa$, where the action is
the following $\phi\cdot n\negmedspace:=\cf(\phi)\cdot n$ for $n\in \cf(\cn)(Z)$
and $\phi\in\sa(X,Z)$. The similar construction is made for
graded right modules and graded bimodules.

\begin{proposition}\label{Prop:HH-HC-equiv}Let $\sa$ and
$\sb$ be superadditive categories and let $\cn$ be a graded
$\sb$-bimodule. If $\cf\negmedspace: \sa\rr\sb$ is an even
equivalence, then there are isomorphisms
of superspaces \begin{align*}\mathsf{(i)}\quad&\HH_{*}(\sa, \cf(\cn))\cong \HH_*(\sb, \cn)&&\HH^{*}(\sa, \cf(\cn))\cong
\HH^*(\sb, \cn).\\
\mathsf{(ii)}\quad&\HC_{*}(\sa)\cong \HC_*(\sb)&&\HC^{*}(\sa)\cong \HC^*(\sb).\end{align*}
\end{proposition}

\begin{proof}The functor $\cf^*: \Gr\sb^e\rr\Gr\sa^e$ induced by $\cf$ is an even equivalence since $\cf$ is itself an even equivalence. Hence the induced functor $\HH^*(\sa,-)$ is a coeffaceable universal $\delta$-functor. Thus the following collection of functors \begin{align*}
\HH^n(\sa, \cf): \Gr\sb^e\rr\Gr K, \quad \cn\mapsto\cf(\cn)\mapsto\HH^n(\sa,\cf(\cn))
\end{align*}
is a universal $\delta$-functor and the collection of functors \begin{align*}\HH^n({\sb},-): \Gr\sb^e\rr\Gr K, \quad \cn\mapsto \HH^n(\sb,\cn)
\end{align*} is a universal $\delta$-functor too.

In order to prove that two universal $\delta$-functors are isomorphic, it is enough to prove that $\HH^0(\sb,\cn)\cong \HH^0(\sa,\cf(\cn))$ as superspaces (see \cite[Section~2.1]{weibel}). If $Y$ is in the image of $\cf$, then we define the natural morphism $$\imath:\quad \HH^0(\sb,\cn)\rr \HH^0(\sa,\cf(\cn)), \quad \cn(Y,Y)_{Y\in\sb}\mapsto\cn(\cf(X),\cf(X))_{X\in\sa}.$$
Since $\cn(Y,Y)_{Y\in\sb}\in\HH^0(\sb,\cn)$, Section~\ref{Point:HH^0} implies that  $n_{Y}\phi=(-1)^{|\phi||n_Y|}\phi n_{Y'}$ for all $\phi\in\sb(Y,Y')$. It follows that
$n_{\cf(X)}\psi= \psi n_{\cf(X')}$ for $\psi\in\sa(X',X)$ and $\imath$ is an isomorphism.

If $Y$ is not in the image of $\cf$, since it is an even equivalence, there exists $X\in\sa$ such that there is a homogeneous isomorphism $h: Y\cong\cf(X)$ with $h\in\sb(\cf(X),Y)$. Thus we get $n_{Y}= hn_{\cf(X)}h^{-1}=h\imath(n_{X})h^{-1}$. Thus $\imath$ is an isomorphism by the above arguments.

The homological case is analogous to cohomological one. If $Y$ is in the image of $\cf$, it follows by applying Remark~\ref{Rem:HH_0} and the following natural well-defined isomorphism
$$\jmath:\quad \HH_0(\sb, \cn)\rr \HH_0(\sa, \cf(\cn)),\quad \overline{\sum_{Y\in\sb}n_{Y}}\mapsto\overline{\sum_{X\in\sb}n_{\cf(X)}}.$$
 If $Y$ is not in the image of $\cf$, since $\cf$ is an equivalence, there exists $X\in\sa$ such that $h: Y\cong \cf(X)$ with homogeneous isomorphism $h\in\sb(\cf(X),Y)$. Using Remark~\ref{Rem:HH_0}, we get $n_{Y}\!=\!h\overline{n_{\cf(X)}}h^{-1}\!=\!h\jmath(n_{X})h^{-\!1}$. As a consequence, the map $\jmath$ is an isomorphism.

 The assertion (ii) follows by applying the Gysin-Connes exact sequence (\ref{Equ:Gysin-Connes-Exact-Sequ}) and (i).
\end{proof}

\begin{proposition}\label{Prop:HH-HC-hat}Let $\sa$ be a superadditive category and $\cm$ a graded $\widehat{\sa}^e$-module. Then \begin{align*}&\HH_*(\sa, \mathrm{hat}\cm)\cong \HH_*(\widehat{\sa},\cm)\quad \text{ and }\quad\HC_*(\sa )\cong \HC_*(\widehat{\sa}),\\ &\HH^*(\sa, \mathrm{hat}\cm)\cong \HH^*(\widehat{\sa},\cm)\quad \text{ and }\quad\HC^*(\sa)\cong \HC^*(\widehat{\sa}).\end{align*}
 \end{proposition}
\begin{proof}Let us define the following collections of functors:
\begin{align*}&\mathrm{hat}_1: \Gr\widehat{\sa}^e\rr \Gr K, \quad \cm\mapsto\mathrm{hat}\cm\mapsto \HH^*(\sa, \mathrm{hat}\cm)\\
&\mathrm{hat}_2: \Gr\widehat{\sa}^e\rr \Gr K, \quad \cm\mapsto \HH^*(\widehat{\sa}, \cm).\end{align*}
Notice that both are universal graded $\delta$-functors (coeffaceable in the homological case and effaceable in the cohomological one).

Now we define the following natural morphisms
\begin{align*}&\eta:  \HH^0(\widehat{\sa}, \cm)\rr \HH^0(\sa,\mathrm{hat}\cm),\qquad
(m_{(X, \epsilon)})_{(X, \epsilon)\in\widehat{\sa}}\mapsto(m_{X})_{X\in\sa};\\
 &\theta:  \HH_0(\widehat{\sa}, \cm)\rr \HH_0(\sa,\mathrm{hat}\cm),\qquad
\overline{\sum_{(X,\epsilon)\in\widehat{\sa}}m_{(X, \epsilon)}}\mapsto
\overline{\sum_{X\in\sa}m_{X}}.
\end{align*}
 It is immediate to see that they are both well-defined and surjective. Moreover, by applying Proposition~\ref{Prop:hat-equiv}, $\eta$ and $\theta$ are isomorphisms.
\end{proof}

\begin{proposition}\label{Prop:HH-HC-mat}Let $\sa$ be a superadditive category and $\cm$ a graded $\widehat{\sa}^e$-module. Then
\begin{align*}&\HH_*(\sa, \mathrm{hat}\cm)\cong \HH_*(\mathrm{mat}\sa,\cm)\quad \text{ and }\quad\HC_*(\sa )\cong \HC_*(\mathrm{mat}\sa),\\ &\HH^*(\sa, \mathrm{mat}\cm)\cong \HH^*(\mathrm{mat}\sa,\cm)\quad \text{ and }\quad\HC^*(\sa)\cong \HC^*(\mathrm{mat}\sa).\end{align*}
\end{proposition}
\begin{proof}Define the following collections of functors:
\begin{align*}&\mathrm{mat}_1: \Gr\mathrm{Mat}(\sa)^e\rr \Gr K, \quad \cn\mapsto\mathrm{mat}(\cm)\mapsto \HH^*(\sa, \mathrm{mat}\cm)\\
&\mathrm{mat}_2: \Gr\mathrm{Mat}\sa^e\rr \Gr K, \quad \cn\mapsto \HH^*(\mathrm{Mat}\sa, \cm).\end{align*}
Notice that both are universal graded $\delta$-functors (coeffaceable in the homological case and effaceable in the cohomological one).

Now we define the following natural morphisms
\begin{align*}&\mu:  \HH^0(\mathrm{Mat}\sa, \cm)\rr \HH^0(\sa,\mathrm{mat}\cm),\qquad
(m_{X_1, \cdots, X_n})_{(X_1, \cdots, X_n)\in\mathrm{Mat}\sa}\mapsto(m_{X})_{X\in\sa};\\
 &\nu:  \HH_0(\mathrm{Mat}\sa, \cm)\rr \HH_0(\sa,\mathrm{mat}\cm),\qquad
\overline{\sum_{(X_1, \cdots, X_n)\in\mathrm{Mat}\sa}(m_{X_1, \cdots, X_n})}\mapsto\overline{\sum_{X\in\sa}m_{X}}.
\end{align*}
 It is immediate to see that they are both well-defined and surjective. Moreover, by applying Section~\ref{Point:Matrix-equ.}, $\mu$ and $\nu$ are isomorphisms.
\end{proof}

The following result extends \cite[Theorem~2.12]{Her-Solotar} and \cite[Theorem~3.6]{Zhao}.

\begin{theorem}\label{Them:HH-HC-Morita}The Hochschild and cyclic (co)homology of superadditive categories are graded Morita equivalent invariants.\end{theorem}
\begin{proof}Thanks to Theorem~\ref{Them:Morita-addi-idem}, it is enough to prove that the Hochschild and cyclic (co)homology are invariant under equivalences, idempotent and additive completion.  By applying Propositions~\ref{Prop:HH-HC-equiv}, \ref{Prop:HH-HC-hat} and \ref{Prop:HH-HC-mat}, we complete the proof.\end{proof}

\section{Shuffles and cyclic shuffles}\label{Sec:Shuflle-Cyclic-Shuffle}
In this section  we first introduce the shuffles of superadditive categories and show that Hochschild homology commute with tensor product. Then we introduce the cyclic shuffles of superadditive categories and determine the K\"{u}nneth exact sequence of  cyclic homology. Throughout this section $\sa$ and $\sb$ are superadditive categories.

\begin{point}{}*Let $\mathfrak{S}_n$ be the symmetric group on $\{1, \cdots, n\}$. A $(p,q)$-\textit{shuffle} is a permutation $\gs\in\mathfrak{S}_{p+q}$ such that \begin{align*}
  \gs(1)<\gs(2)<\cdots<\gs(p)\quad\text{ and }\quad \gs(p+1)<\gs(p+2)<\cdots<\gs(p+q).
\end{align*}
For any superadditive category $\sa$ we let $\fs$ act on the left on $\sn_{n+1}=\sn_{n+1}(\sa)$ by \begin{align*}
  \gs(a_0,a_1,\cdots, a_n)=(a_0, a_{\gs^{-1}(1)}, \cdots, a_{\gs^{-1}(n)}).
\end{align*}
Thus if $\gs$ is a $(p,q)$-shuffle the elements $\{a_1, a_2, \cdots, a_p\}$ appear in the same order in the sequence $\gs\cdot(a_0,a_1,\cdots, a_n)$ as do the element $\{a_{p+1}, a_{p+2}, \cdots, a_{p+q}\}$.

 The {\it shuffle product}
\begin{align*}
  \mathrm{sh}_{pq}: \sn_{p}(\sa)\otimes\sn_q(\sb)\rr\sn_{p+q}(\sa\otimes\sb)
\end{align*}
is defined by the following formula:
\begin{align*}
  \mathrm{sh}_{pq}((a_0, a_1, \cdots, a_p)\otimes(b_0, b_1, \cdots, b_q))\!=\!\sum_{\gs}\mathrm{sgn}(\gs)\gs(a_0\!\otimes\! b_0, a_1\!\otimes\! 1, \cdots, a_p\!\otimes\!1; 1\!\otimes\! b_1, \cdots, 1\!\otimes\! b_q),
\end{align*}
where the sum is extended over all $(p,q)$-shuffles.
Note that the same formula defines more generally a shuffle product from $\sn_p(\sa,\cm)\otimes\sn_q(\sb,\cn)$ to $\sn_{p+q}(\sa\otimes\sb, \cm\otimes\cn)$.
\end{point}

\begin{proposition}\label{Prop:HH-shuffle}For $a\in\sn_p(\sa)$ and $b\in\sn_q(\sb)$, we have
  \begin{align*}\partial_{p+q}\circ\mathrm{sh}_{pq}(a\otimes b)=\sh_{p-1q}(\partial_p(a)\otimes b)+(-1)^{|q|}\sh_{pq-1}(a\otimes \partial_q(b)).\end{align*}
\end{proposition}
\begin{proof}Let $a=(a_0,a_1, \cdots, a_p)$ and $b=(b_0,b_1, \cdots, b_q)$ and write $$\sh_{pq}(a\otimes b)=\sum\pm(c_0=a_0\otimes b_0, c_1, \cdots, c_{p+q}),$$ where $c_i$ is either in $I_1=\{a_1\otimes 1, \cdots, a_p\otimes 1\}$ or in $I_2=\{1\otimes b_1, \cdots, 1\otimes b_q\}$. Fix $i$ ($0\leq i\leq n=p+q$) and consider the element $d_n^i(c_0,c_1, \cdots, c_n)$ appearing in the expansion of $\partial_n(\sh_{pq}(a,b))$. If $c_i$ and $c_{i+1}$ are in $I_1$ (resp. $I_2$), or if $i=0$ and $c_1$ is in $I_1$ (resp. $I_2$), then $d_n^i(c_0,c_1, \cdots, c_n)$ appears also in the expansion of $\sh_{p-1q}(\partial_{p+q}(a),b)$ (resp. $\sh_{pq-1}(a,\partial_q(b))$ and conversely. If $c_i$ and $c_{i+1}$ belong to two different sets, then $(c_0, c_1, \cdots, c_{i-1}, c_{i+1},c_i, c_{i+2},\cdots, c_{p+q})$ is also a shuffle and appears in the expansion of $\sh_{pq}(a\otimes b)$. As its sin is the opposite of the sign (in front) of $(c_0, c_1, \cdots, c_{p+q})$, these two elements cancel after applying $d_n^i$ (because $c_ic_{i+1}=c_{i+1}c_i$).
Thus we have proved the proposition.\end{proof}

Recall that the tensor product of complexes $(C_{\bullet},\partial)$ and $(\widetilde{C}_{\bullet},\partial)$ is the complex $(C\otimes \widetilde{C})_{\bullet}$ with $(C\otimes \widetilde{C})_{n}=\oplus_{p+q=n}C_p\otimes \widetilde{C}_q$ and the differential map is defined by the formula $$d(x\otimes y)=(\partial\otimes 1+1\otimes \partial)(x\otimes y)=\partial(x)\otimes y+(-1)^{|x|}x\otimes\partial(y).$$
Let \begin{align*}
  \sh: (\sn_*(\sa)\otimes\sn_*(\sb))_n=\bigoplus_{p+q=n}\sn_p(\sa)\otimes\sn_q(\sb)\rr\sn_{n}(\sa\otimes\sb)
\end{align*}
be the sum of the shuffle product maps $\sh_{pq}$ for $p+q=n$.

The following fact follows readily form Proposition~\ref{Prop:HH-shuffle}.
\begin{corollary}\label{Cor:partial-shuffle}
 The map $\sh: \sn_*(\sa)\otimes\sn_*(\sb)\rr\sn_*(\sa\otimes\sb)$ is a map of complexes of degree of zero, that is, $[\partial,\sh]:=\partial\circ\sh-\sh\circ(\partial\otimes 1+1\otimes \partial)=0$.
\end{corollary}

The following is a categorical version of the Eilenberg-Zilber Theorem for Hochschild homology of algebras.
\begin{theorem}\label{Them:HH-shuffle}Assume that $K$ is a field. Then the shuffle map $\sh$ induces an isomorphism
\begin{align*}
  \sh_*: \HH_*(\sa)\otimes\HH_*(\sb)\stackrel{\sim}\longrightarrow \HH_*(\sa\otimes\sb).
\end{align*}
\end{theorem}
\begin{proof}The theorem can be proved by applying the same arguments as that of the Eilenberg-Zilber Theorem for Hochschild homology of algebras (see \cite[Chapter 8, Theorem~8.1]{Maclane}. \end{proof}

\begin{point}{}* By definition a $(p,q)$-\textit{cyclic shuffle} is a permutation $\gs\in\fs_{p+q}$ defined as follows: perform a cyclic permutation of any order on the set of $\{1, \cdots, p\}$ and a cyclic permutation of any order on the set $\{p+1, \cdots, p+q\}$.  Then shuffle the two results to obtain the permutation. This is a cyclic shuffle if $1$ appears before $p+1$ in the sequence $(\sigma(1), \cdots, \sigma(p+q))$.

The cyclic shuffle is a map $\perp: \sn_{p}(\sa)\otimes\sn_q(\sb)\rr\sn_{p+q}(\sa\otimes\sb)$ given by\begin{align*}
  (a_0, a_1, \cdots, a_p)\perp(b_0, b_1, \cdots, b_q))=\sum_{\gs}\mathrm{sgn}(\gs)\gs(a_0\otimes b_0, a_1\otimes 1, \cdots, a_p\otimes1; 1\otimes b_1, \cdots, 1\otimes b_q),
\end{align*}
where the sum is extended over all $(p,q)$-cyclic shuffles. Note that there is a similar operation $\perp: \sn_{p}(\sa)\otimes\sn_q(\sb)\rr\sn_{n}(\sa)\otimes\sn_n(\sb)$ for $n=p+q$.

The \textit{cyclic shuffle map} is a map of degree 2 \begin{align*}\csh_{pq}: \sn_{p}(\sa)\!\otimes\!\sn_q(\sb)\rr\sn_{p+q+2}(\sa)\!\otimes\!\sn_{p+q+2}(\sb)\end{align*} defined by the following formula:
\begin{align*}
  \mathrm{csh}_{pq}(a,b):=\csh_{pq}(a\otimes b)=s(a)\perp s(b),
\end{align*}
where $s: \sn_n(\sa)\rr\sn_{n+1}(\sa)$ is the extra degeneracy of $\sn(\sa)$  (cf.~Section~\ref{Point:Cyclic-bicomplex}).  Note that $sB=0$ in the normalized setting, it follows that for any $a$ and $b$, \begin{equation}\label{Equation:B-s-csh}
  \csh(B(a),b)=sB(a)\perp b\qquad\text{and}\qquad \csh(a, B(b))=0.
\end{equation}
\end{point}
\begin{proposition}\label{Prop:HC-shuffle}
  For $a\!\in\!\sn_{p}(\!\sa\!)$ and $b\!\in\!\sn_{q}(\!\sb\!)$, the following equality holds in $\sn_{p\!+\!q\!+\!1}(\!\sa\!\otimes\!\sb\!)$: \begin{align*}
    B\sh_{pq}(a,b)\!-\!\left(\!\sh(B(a),b)\!+\!(\!-1\!)^{|a|}\sh(a,B(b))\!\right)\!
    +\!\partial\csh(a,b)\!-\!\csh(\partial(a),b)\!-\!(\!-\!1)^{|a|}\csh(a, \partial(b))\!=\!0.
  \end{align*}
\end{proposition}
\begin{proof}Let $a=(a_0,a_1, \cdots, a_p)$ and $b=(b_0,b_1, \cdots, b_q)$. Note that the image  of $a\otimes b$ under any of these compositions of these maps is the sum of elements of two differential types: either it is a permutation of $(a_0\otimes 1, \cdots, a_{p}\otimes 1,1\otimes b_0,\cdots, 1\otimes b_q)$, or it is an element of the form $(1\otimes 1, \cdots)$ where one the of the other entries is of the form $a-i\otimes a_j$. We only need to show that the sum of the elements of the first type is 0 (similarly we can show that the sum of the elements of the second type is 0).

Elements of the first type arise only form $\partial\csh_{pq}$, $\sh_{p+1q}(B\otimes1)$ and $\sh_{pq+1}(1\otimes B)$. Let $a=(a_1, \cdots, a_n)$ and $\tau$ the cyclic permutation such that $\partial(1,a)=a+\mathrm{sgn}(\tau)\tau(a)$ modulo the elements of the form $(1, \cdots)$. So the permutations coming from $\partial\csh_{pq}$ are of the form (cyclic shuffle) or $\tau\circ$(cyclic shuffle). The cyclic shuffle which have 1 in the first position cancel with the permutations coming from $\sh_{p+1q}(B\otimes 1)$. For such a $\sigma$, consider $\tau\sigma$. If $p+1$ is in the first position, then $\tau\sigma$ cancels with a permutation coming from $\sh_{P+1q}(B\otimes 1)$, if not then it is a cyclic shuffle $\sigma'$ (and it cancels with it). So we are led to examine $\tau\sigma'$ for which we play the same game. By the end all the elements have disappeared.\end{proof}
\begin{point}{}*The \textit{cyclic shuffle product}
\begin{align*}
  \csh: (\sn(\sa)\otimes\sn(\sb))_n=\bigoplus_{p+q=n}\sn_p(\sa)\otimes\sn_q(\sb)\lrr\csh(\sn(\sa),\sn(\sb))_{n+2}
\end{align*}
is the sum of all the $(p,q)$-shuffle maps $\csh_{pq}$ for $p+q=n$. \end{point}

\begin{lemma}\label{Lem:partial-B-sh-csh}
  The maps $\partial$, $B$, $\sh$, and $\csh$ satisfy the following formulas in the normalized setting \begin{align*}
    \mathsf{(i)}\quad [\partial,\sh]=0,\qquad\mathsf{(ii)}\quad [B,\sh]+[\partial, \csh]=0,\qquad\mathsf{(iii)}\quad [B,\csh]=0.
  \end{align*}
\end{lemma}

\begin{proof}$\mathsf{(i)}$ is exactly Corollary~\ref{Cor:partial-shuffle}. $\mathsf{(ii)}$ is an immediate consequence of Proposition~\ref{Prop:HC-shuffle}. $\mathsf{(iii)}$ follows directly by applying Equation~\ref{Equation:B-s-csh}.
\end{proof}

\begin{theorem}\label{Them:Shuffle-cyclic-shuffle}
  Let $\sa$ and $\sb$ be superadditive categories. Then the shuffle product and the cyclic shuffle product induce a canonical isomorphism \begin{align*}
    \mathrm{Sh}: \HC_*(\sn_*(\sa)\otimes\sn_*(\sb))\stackrel{\sim}\lrr\HC(\sh(\sn_*(\sa\otimes\sb))).
  \end{align*}
  Furthermore $\mathrm{Sh}$ commutes with the morphisms $B$, $I$ and $S$ of Connes's exact sequence.
\end{theorem}

\begin{proof} Lemma~\ref{Lem:partial-B-sh-csh}(i) implies that the shuffle map $\sh: \sn_*(\sa)\otimes \sn_*(\sb)\rr\sn_*(\sa\otimes\sb)$ is a map of complexes. However it is not a map of mixed complexes since $B$ does not commute with $\sh$ (see Lemma~\ref{Lem:partial-B-sh-csh}(ii)). Note that there is following commutative diagram
$$\xymatrix{
   0 \ar[r]^{} & \sn_*(\sa)\otimes
   \sn_*(\sb) \ar[d]_{\sh} \ar[r]^{} & \mathrm{Tot}\sn_*(\sa)\otimes
   \sn_*(\sb) \ar@{.>}[d]_{\mathrm{Sh}} \ar[r]^{} & \mathrm{Tot}\sn_*(\sa)\otimes
   \sn_*(\sb)[2] \ar@{.>}[d]^{\mathrm{Sh}[2]} \\
    0\ar[r]^{} & \sn_*(\sa\otimes\sb)\ar[r]^{} & \mathrm{Tot}\sn_*(\sa\otimes\sb)\ar[r]^{} & \mathrm{Tot}\sn_*(\sa\otimes\sb)[2].}$$
Recall that $$\aligned
  & \mathrm{Tot}(\sn_*(\sa)\otimes
   \sn_*(\sb))_n=\bigoplus_{k=0}^{\infty}(\sn_*(\sa)\otimes
   \sn_*(\sb))_{n-2k}\\& \mathrm{Tot}(\sn_*(\sa\otimes
  \sb))_n=\bigoplus_{k=0}^{\infty}(\sn_*(\sa)\otimes
   \sn_*(\sb))_{n-2k}
\endaligned$$
Now let $$\mathrm{Sh}=\left(\begin{array}
  {ccccc}\sh&\csh&&&\\&\sh&\csh&&\\&&\sh&\csh&\\&&&\ddots&\ddots
\end{array}\right).$$
By Lemma~\ref{Lem:partial-B-sh-csh}, $\mathrm{Sh}$ is a morphism of complexes. The commutativity of the right-hand square comes form the form of $\mathrm{Sh}$ and the commutativity of the left-hand square is immediate. Applying Theorem~\ref{Them:HH-shuffle}, we complete the proof of the theorem.
\end{proof}

We now state our main theorem comparing $\HC_n(\sa\otimes\sb)$ with $\HC_n(\sa)$ and $\HC_n(\sb)$.

\begin{theorem}\label{Them:HC-Cyclic-shuffle}
  Let $\sa$ and $\sb$ be superadditive categories over a filed $K$. Then there exists a natural long exact sequence
  \begin{align*}
    \xymatrix@C=0.5cm{ \cdots \ar[r]&
\HC_n(\sa\otimes\sb)\ar[d]^{}\\  &\displaystyle\bigoplus_{p+q=n}\HC_p(\sa)\otimes\HC_q(\sb)\quad\ar[rr]^{S\otimes \mathrm{id}-\mathrm{id}\otimes S}&&\quad\displaystyle\bigoplus_{p+q=n-2} \HC_p(\sa)\otimes\HC_q(\sb) \ar[d]^{}\\ &&&\HC_{n-1}(\sa\otimes\sb) \ar[r]^{}&\cdots}
  \end{align*}
\end{theorem}

\begin{proof}The proof follows direct by applying Theorem~\ref{Them:Shuffle-cyclic-shuffle}, Lemma~\ref{Lem:partial-B-sh-csh}, and \cite[Lemma~4.39]{Loday}.
\end{proof}
\bibliographystyle{amsplain}

\end{document}